\documentclass[a4paper]{amsart}
\usepackage[mathscr]{eucal}
\usepackage{amssymb}
\usepackage{latexsym}
\usepackage{enumerate}

\theoremstyle{plain}
\newtheorem{thm}{Theorem}[section]
\newtheorem{prop}[thm]{Proposition}
\newtheorem{lem}[thm]{Lemma}
\newtheorem{cor}[thm]{Corollary}

\theoremstyle{definition}
\newtheorem{defn}[thm]{Definition}
\newtheorem{rem}[thm]{Remark}

\begin{document}

\title[Universal lattices]{Fixed point property for universal lattice on Schatten classes}
\author{Masato MIMURA}
\address{Graduate School of Mathematical Sciences, 
University of Tokyo, Komaba, Tokyo, 153-8914, Japan/ \'{E}cole Polytechnique f\'{e}d\'{e}rale de Lausanne, SB--IMB--EGG, 
Station 8, B\^{a}timent MA, Lausanne, Vaud, CH-1015, Switzerland}
\email{mimurac@ms.u-tokyo.ac.jp}
\thanks{The author is 
supported by JSPS Research Fellowships for Young Scientists No.20-8313.}


\begin{abstract}
The special linear group $G=\mathrm{SL}_n (\mathbb{Z}[x_1, \ldots , x_k])$ ($n$ at least $3$ and $k$ finite) is called the universal lattice. Let $n$ be at least $4$, $p$ be any real number in $(1, \infty)$. The main result is the following: any finite index subgroup of $G$ has the fixed point property with respect to every affine isometric action on the space of $p$-Schatten class operators. It is in addition shown that higher rank lattices have the same property. These results are generalization of previous theorems repsectively of the author and of Bader--Furman--Gelander--Monod, which treated commutative $L^p$-setting.
\end{abstract}

\maketitle

\section{Introduction and main results}\label{sec:intro}

In this paper, \textit{every group is assumed to locally compact and second countable, and every subgroup of a} (\textit{topological}) \textit{group is assumed to be closed}. The symbol $k$ is used for representing any finite natural (positive) number, and we set $A=\mathbb{Z}[x_1,\ldots ,x_k]$. Hereafter, by ``\textit{higher rank lattices}" we mean lattices (, namely, discrete subgroup with finite covolume) in semisimple algebraic groups over local fields (possibly archimedean) with each factor having local rank$\geq2$. We use the symbol $\langle \cdot \mid \cdot \rangle$ for the inner product on a Hilbert space.

The special linear group $G=\mathrm{SL}_n (A)$$=\mathrm{SL}_n (\mathbb{Z}[x_1, \ldots , x_k])$ (where $n\geq 3$) is called the \textit{universal lattice} by Y. Shalom in \cite{Shal1}. It was a long standing problem to determine whether this group satisfies \textit{Kazhdan's property} $(T)$ (\cite{Ka}), and finally Shalom and L. Vaserstein  has answered this problem in the affirmative. 
\begin{thm}$($Shalom \cite{Shal4}, Vaserstein \cite{Vas}$)$\label{thm:SV}
Universal lattices $G=\mathrm{SL}_n (\mathbb{Z}[x_1, \ldots , x_k])$ $(n\geq 3)$ have property $(\mathrm{T})$.
\end{thm}
Because property (T) passes to group quotients, this result immediately implies groups such as $\mathrm{SL}_{n\geq 3}(\mathbb{Z}[1/p])$ (here $p$ is a prime number); $\mathrm{SL}_{n\geq 3}(\mathbb{Z}[\sqrt{2},\sqrt{3}])$; $\mathrm{SL}_{n\geq 3}(\mathbb{F}_{q}[x])$ ($\mathbb{F}_{q}$ is the field of order $q$ and $q$ is a prime power); and $\mathrm{SL}_{n\geq 3}(\mathbb{Z}[t, t^{-1}])$ have property (T). Note that in four examples above, all but last one are higher rank (hence arithmetic (or $S$-arithmetic)) lattices. Kazhdan's original result \cite{Ka} states that higher rank lattices (recall the notation in this paper from the first paragraph) has property (T). Hence in these cases property (T) is classical. However the last one in the examples above cannot be realized as a higher rank (hence arithmetic) lattice because it contains an infinite normal subgroup of infinite index, and property (T) for this group had not been obtained before the Shalom--Vaserstein theorem. Therefore property (T) for universal lattices can be regarded as a \textit{non-arithmetization} of extreme rigidity for higher rank lattices. 
Recall that the celebrated Delorme--Guichardet theorem, see \cite{BHV}, states property (T) is equivalent to  \textit{property }$(FH)$, which is defined as follows: a group $\Lambda$ is said to have \textit{property} $(FH)$ if every (continuous) affine isometric action of $\Lambda$ on a Hilbert space has a global fixed point.

From this point, \textit{we always assume }$p\in (1, \infty)$. In 2007, Bader, Furman, Gelander, and Monod considered fixed point properties in much wider framework, and they defined the fixed point property, \textit{property} $(F_{\mathcal{B}})$, for a family $\mathcal{B}$ of Banach spaces (, or a single Banach space $B$). They in particular paid heavy attention to the case of $\mathcal{B}=\mathcal{L}^p$, which denotes the family of all $L^p$-spaces, because property $(\mathrm{F}_{\mathcal{L}^p})$ turns out to be much stronger than property $(\mathrm{T})$ ($\Leftrightarrow$ property $(\mathrm{FH})=(\mathrm{F}_{\mathcal{L}^2})$), provided that $p\gg 2$. Indeed, P. Pansu \cite{Pa} shown the group $\mathrm{Sp}_{n,1}$, which has property (T) if $n\geq 2$, fails to have property $(\mathrm{F}_{\mathcal{L}^p})$ if $p>4n+2$. Moreover G. Yu \cite{Yu} shown that every hyperbolic group $H$, including wide range of groups with property (T), has $p>2$ such that it admits a metrically proper (hence far from having a global fixed point) affine isometric action on $l^p$-space. In comparison, one of the main results of Bader--Furman--Gelander--Monod in \cite{BFGM} is the following:
\begin{thm}$($\cite{BFGM}$)$\label{thm:BFGM}
Higher rank lattices, in the notation of this paper, remain to have property $(\mathrm{F}_{\mathcal{L}^p})$ for all $p \in (1,\infty )$.
\end{thm}

 In \cite{Mim}, the author extended the Shalom--Vaserstein theorem and obtained the following ``non-arithmetization" of Theorem~\ref{thm:BFGM}:
\begin{thm}$($\cite{Mim}$)$\label{thm:prev}
Let $n\geq 4$. Then for any $p\in (1,\infty)$ the universal lattice $G=\mathrm{SL}_n (\mathbb{Z}[x_1, \ldots , x_k])$ has property $(\mathrm{F}_{\mathcal{L}^p})$.
\end{thm}
Note that property $(\mathrm{F}_{\mathcal{L}^p})$ is of high importance also from the view point of group actions on the circle, see \cite[\S 4]{Nav}. Also we
mention in Theorem~\ref{thm:prev}, the case of $n=3$ seems to remain open.

The main result of this paper is to extend this result to that for the family $\mathcal{C}_p$ of the spaces $C_p$ of $p$-Schatten class operators on any separable Hilbert space. Here for a separable Hilbert space $\mathfrak{H}$, a bounded linear operator $ a \in \mathbb{B}(\mathfrak{H})$ is said to be of $p$\textit{-Schatten} class ($a\in C_p$) if $\mathrm{Tr}(|a|^p)<\infty$ holds, where $\mathrm{Tr}$ is the canonical (semifinite) trace: take an orthonomal basis $(\xi_i)_{i\in \mathbb{N}}$ for $\mathfrak{H}$ and for $t\in \mathbb{B}(\mathfrak{H})^+$, set
$$
\mathrm{Tr}(t):= \sum_{i\in \mathbb{N}}\langle t\xi_i \mid \xi_i \rangle \in [0,\infty]
$$
(, which is independently determined of the choice of an orthonormal basis of $\mathfrak{H}$); and $|a|:=(a^{*}a)^{1/2}$ is the absolute value of $a$. It can be regarded as a some generalization of Theorem~\ref{thm:prev} to noncommutative $L^p$-setting. Motivating fact on this study is that fixed point property on $\mathcal{C}_p$ has potential for some application to group actions on higher dimensional manifolds.  

Furthermore, in this paper we consider stronger property than property $(\mathrm{F}_B)$, called \textit{property }$(FF_B)$. Recall every affine isometric action $\alpha\colon G \curvearrowright B$ can be written as $\alpha (g)\cdot \xi$$=\rho (g)\xi +b(g)$, where $\rho$ is a isometric linear representation and $b\colon G\to B$ is a $\rho$\textit{-}($1$\textit{-})\textit{cocycle}, namely, for any $g,h\in G$, $b(gh)=b(g)+\rho(g)b(h)$ holds. 
We consider a ``quasification" of cocycles, namely we allow uniformly bounded error from being cocycles. Property $(\mathrm{FF}_B)$ is the boundedness property for any quasi-cocycle into every isometric representation in $B$. Property $(\mathrm{FF}_B)/\mathrm{T}$ is a weaker form of $(\mathrm{FF}_B)$ and that asserts the boundedness of quasi-cocycles \textit{modulo trivial linear part}. Now we state our main result:

\begin{thm}\label{thm:Cp}
Let $n\geq 4$. Then for any $p\in (1,\infty)$ the universal lattice $G=\mathrm{SL}_n(\mathbb{Z}[x_1, \ldots , x_k])$ has property $(\mathrm{F}_{\mathcal{C}_p})$. Equivalently, every affine isometric action of $G$ on the space  $C_p$ of $p$-Schatten class operators $($on any separable Hilbert space$)$ has a global fixed point. 
Furthermore, for any $p\in (1,\infty)$, $G$ has property $(\mathrm{FF}_{\mathcal{C}_p})/ \mathrm{T}$ $($``property $(\mathrm{FF}_{\mathcal{C}_p})$ modulo trivial part"$)$. In particular, for any isometric linear representation $\rho$ on $C_p$ $\mathrm{which}$ $\mathrm{satisfies}$ $\rho \not\supseteq 1_G$, every quasi-$\rho$-cocycle is bounded. 

Both property $(\mathrm{F}_{\mathcal{C}_p})$ and property $(\mathrm{FF}_{\mathcal{C}_p})/\mathrm{T}$ remain valid by taking group quotients and taking finite index subgroups of $G$ above. 
\end{thm}
For precise definitions of property $(\mathrm{FF}_B)$ and property $(\mathrm{FF}_B)/\mathrm{T}$, see Section~\ref{sec:criter}. In the scope of the author, it seems unknown at the moment, whether quasi-cocycles on universal lattices into the \textit{trivial} representation must be bounded. Neither does the author know whether $(\mathrm{FF}_B)/\mathrm{T}$ is strictly weaker than $(\mathrm{FF}_B)$.

Therefore for any commutative and finitely generated ring $R$ (we always assume a ring $R$ is unital and associative), the following holds: the \textit{elementary group} $\mathrm{E}_{n\geq 4}(R)$ and finite index subgroups therein have $(\mathrm{F}_{\mathcal{C}_p})$, and property $(\mathrm{FF}_{\mathcal{C}_p})/\mathrm{T}$. Here the \textit{elementary group} over $R$ is the multiplicative group in $n\times n$ matrix ring $M_n(R)$ generated by \textit{elementary matrices}. An \textit{elementary matrix} in $M_n(R)$ is an $n\times n$ matrix whose entries are $1$ on diagonal and all but one entries off diagonal are $0$. The Suslin stability theorem (\cite{Sus}) states for $A=\mathbb{Z}[x_1,\ldots ,x_k]$ $\mathrm{E}_n(A)$ coincides with $\mathrm{SL}_n(A)$ provided $n\geq 3$, whereas $\mathrm{E}_2(A)$ is a proper subgroup of $\mathrm{SL}_2(A)$ (\cite{Coh}). 

Property $(\mathrm{F}_{B})$ and Property $(\mathrm{FF}_{B})$ have natural interpretation in terms of (ordinary and bounded) group cohomology. Thus by Theorem~\ref{thm:Cp}, we have the following corollary. Here \textit{bounded cohomology} is defined by requesting additional condition that every cochain has bounded range, and the map $\Psi^2$ in below is induced by the natural inclusion from bounded to ordinary cochain complexes ($\Psi^2$ is called the \textit{comparison map} in degree $2$).  For details of bounded cohomology with Banach coefficients, see \cite{Mon}, \cite{BM1}, and \cite{BM2}.

\begin{cor}\label{cor}
Let $n\geq 4$ and $R$ is a $($unital, associative,$)$ commutative and finitely generated ring. Then for any $p\in (1,\infty)$, every finite index subgroup $\Gamma$ in $\mathrm{E}_n(R)$ satisfies the following: 
\begin{itemize}
   \item  for any isometric linear representation $\rho$ on $C_p$, 
   $H^1(\Gamma ;C_p, \rho)=0$;
   \item for any isometric linear representation $\rho$ on $C_p$ $\mathrm{which}$ $\mathrm{satisfies}$ $\rho \not\supset 1_{\Gamma}$, the natural map from second bounded cohomology to second cohomology: 
   $$\Psi^2 \colon H^2_{\mathrm{b}}(\Gamma ; C_p, \rho) \to H^2(\Gamma ; C_p, \rho)$$
   is injective.
\end{itemize}
\end{cor}
The latter item follows from that the kernel of $\Psi^2$ above is naturally isomorphic to the following space:
$
\{ \textrm{quasi-$\rho$-cocycles}\}/(\{ \textrm{$\rho$-cocycles}\}+\{ \textrm{bounded maps}\}).
$

Our proof of Theorem~\ref{thm:Cp} for universal lattices (and group quotients) consists of two steps: in the first step, in Section~\ref{sec:criter} we show certain criteria for a family $\mathcal{B}$ of Banach spaces with respect to which universal lattices satisfy property $(\mathrm{F}_{\mathcal{B}})$ and property $(\mathrm{FF}_{\mathcal{B}})/\mathrm{T}$. To the best of  the author's knowledge, no such criteria had been observed. Since our criteria seems to be of their own interest and importance, we state it here. For the definition of (relative) property $(\mathrm{T}_{\mathcal{B}})$), and of \textit{ultraproducts} of Banach spaces, we refer to Section~\ref{sec:criter}.
\begin{thm}$($criteria for fixed point properties for universal lattices$)$\label{thm:criter}
Set $A=\mathbb{Z}[x_1,\ldots , x_k]$. Let $\mathcal{B}$ be a family of superreflexive Banach spaces. Suppose either of the following two conditions is fulfilled:
\begin{enumerate}[$($$i$$)$]
   \item The pair $\mathrm{E}_2(A) \ltimes A^2 \trianglerighteq A^2$ has relative property $(\mathrm{T}_{\mathcal{B}})$; and $\mathcal{B}$ is stable under ultraproducts.
   \item As properties for countable discrete groups, relative property $(\mathrm{T})$ implies relative property $(\mathrm{T}_{\mathcal{B}})$: that means; for any pair of a countable discrete group $\Lambda$ and a normal $($not necessarily  proper$)$ subgroup $\Lambda_0$ therein  whenever the pair $\Lambda \trianglerighteq \Lambda_0$ has relative property $(\mathrm{T})$, it has relative property $(\mathrm{T}_{\mathcal{B}})$.
\end{enumerate}
Then for any $n\geq 4$, the universal lattice $\mathrm{SL}_n(A)$ possesses property $(\mathrm{F}_{\mathcal{B}})$ and furthermore possesses property $(\mathrm{FF}_{\mathcal{B}})/\mathrm{T}$.
\end{thm}
Note that ``condition $(ii)$ implies the conclusion" deeply relies on Theorem~\ref{thm:SV}. Also, if condition $(ii)$ is satisfied, then relative property $(\mathrm{T}_{\mathcal{B}})$ for $\mathrm{E}_2(A) \ltimes A^2 $$\trianglerighteq A^2$ is automatic because relative property $(\mathrm{T})$ for that pair was shown in \cite{Shal1} much earlier. Here by $\Lambda:= \mathrm{E}_2(A) \ltimes A^2 $$\trianglerighteq A^2=:\Lambda_0$, we mean 
$$
\Lambda = \left\{ (Z,\zeta)=
 \left( 
\begin{array}{c|c}
Z & \zeta \\
\hline 
0\ 0 &  1 
\end{array}
\right)
 : Z \in \mathrm{E}_{2} (A), \, \zeta \in A^{2}
 \right\} \trianglerighteq \{ (I_2,\zeta ): \zeta\in A^2 \} =\Lambda_0. 
$$
For some examples of $\mathcal{B}$ which satisfy condition $(i)$ or condition $(ii)$, see Section~\ref{sec:concl} and the last part of Section~\ref{sec:criter}.

In the second step, we verify that the family $\mathcal{C}_p$ indeed fulfills condition $(ii)$:
\begin{prop}\label{prop:C_p}
For any $p\in (1,\infty)$ relative property $(\mathrm{T})$ implies  relative property $(\mathrm{T}_{\mathcal{C}_p})$ among locally compact and second countable groups. 
\end{prop}
This part is inspired by a work of M. Puschnigg \cite{Pu}, who extended a method of Bader--Furman--Gelander--Monod for the case of commutative $L^p$-spaces on $\sigma$-finite measures.

For the proof of Theorem~\ref{thm:Cp} for finite index subgroups of universal lattices, we need the $p$-induction theory for (quasi)-1-cocycles of Bader--Furman--Gelander--Monod. We also need to deal with noncommutative $L^p$-space associated with the von Neumann algebra $L^{\infty}(\mathcal{D})\otimes \mathbb{B}(\mathfrak{H})$ (with the canonical trace), where $\mathcal{D}$ is a finite measure space. We examine them in Section~\ref{sec:fi}. The arguments above in addition implie the following:
\begin{thm}\label{thm:highrank}
Higher rank lattices, in the notation of this paper, have property $(\mathrm{F}_{\mathcal{C}_p})$ for all $p \in (1,\infty )$. More precisely, the following holds true: Let $G=\Pi _{i=1}^{m}\mathbf{G}_i(k_i)$, where $k_i$ are local field, $\mathbf{G}_i(k_i)$ are $k_i$-points of Zariski connected simple $k_i$-algebraic group $\mathbf{G}_i$. Assume each simple factor $\mathbf{G}_i(k_i)$ has $k_i$-rank at least $2$. Then $G$ and lattices therein have property $(\mathrm{F}_{\mathcal{C}_p})$ for any $p\in (1,\infty)$.
\end{thm}
This result is a noncommutative analogue of Theorem~\ref{thm:BFGM}, and can be seen a generalization of (a part of) a work of Puschnigg \cite[Corollary 5.10]{Pu}. For noncommutative $L^p$-spaces associated with semifinite von Neumann algebra other than $C_p$ (associated with $(\mathbb{B}(\mathfrak{H}),\mathrm{Tr})$), see Remark~\ref{rem:typeII}.

\bigskip

\textit{Organization of this paper}: Section~\ref{sec:criter} is devoted to basic definitions and proof of Theorem~\ref{thm:criter}. In Section~\ref{sec:Cp}, we prove Proposition~\ref{prop:C_p}, and thus prove Theorem~\ref{thm:Cp} for universal lattices. In Section~\ref{sec:fi}, we consider $p$-inductions and generalization of Proposition~\ref{prop:C_p}. By these, we complete the proof of Theorem~\ref{thm:Cp} and Theorem~\ref{thm:highrank}. Corollary~\ref{cor} is immediate from Theorem~\ref{thm:Cp} and a mere interpretation (see \cite{BHV}, \cite{BFGM}, \cite{BM1}, \cite{BM2}, and \cite{Mon}), so that we will not exhibit a proof of that. In Section~\ref{sec:concl}, we make some concluding remarks on condition $(i)$ in Theorem~\ref{thm:criter}.

\section{Property $(\mathrm{T}_B)$, $(\mathrm{F}_B)$, $(\mathrm{FF}_B)$, $(\mathrm{FF}_B)/\mathrm{T}$; and our criteria}\label{sec:criter}
Hereafter, \textit{we assume all Banach spaces }$B$ \textit{in this paper are superreflexive.} This condition is equivalent to that $B$ has a compatible norm to a uniformly convex and uniformly smooth norm. Basic example is any $L^p$-space with $p\in (1,\infty)$. We refer to \cite[\S A]{BL} for a comprehensive study of this topic. In this section, $\Lambda$ is a locally compact second countable group, $\rho$ is a (continuous) isometric linear representation of $\Lambda$ in $B$. Also every subgroup of $\Lambda$ is assumed to be closed. 

Here we recall some basic results from \cite[\S 2]{BFGM}: first, thanks to superreflexivity of $B$, there exists a uniformly convex and uniformly smooth norm on $B$, compatible to the original norm, with respect to which $\rho$ is still isometric. Hence we can assume $B$ is uniformly convex and uniformly smooth. Second, then by the uniform smoothness, there is a natural complement $B'_{\rho(\Lambda)}$  of $B^{\rho(\Lambda)}$ in $B$ which is a $\rho(\Lambda)$ invariant space. Here $B^{\rho(\Lambda)}$ is the subspace of $B$ of $\rho(\Lambda)$-invariant vectors. Also, $B'_{\rho(\Lambda)}$ is isomorphic to $B/B^{\rho(\Lambda)}$ as $\rho(\Lambda)$ representation spaces. Precisely, $B'_{\rho(\Lambda)}$ is the annihilator of the subspace $(B^{*})^{\rho^{\dagger}(\Lambda)}\subseteq B^*$ in $B$, where $\rho^{\dagger}$ denotes the contragradient representation: for the duality $\langle \cdot ,\cdot \rangle:B\times B^* \to \mathbb{C}$ and for any $g\in \Lambda$, any $x\in B$, and any $\phi\in B^*$, 
$
\langle x,\rho^{\dagger}(g)\phi \rangle:=\langle \rho(g^{-1})x,\phi \rangle$. 
Finally, by uniform convexity of $B$, for any (continuous) affine isometric action, the existence of a $\Lambda$-fixed point is equivalent to boundedness of some (, equivalently, any) $\Lambda$-orbit. By considering orbit of the origin $0\in B$, this means for any $\rho$ on $B$, a $\rho$(1-)-cocycle $b$ is a coboundary if and only if it has bounded range. 

Now we recall definitions of property $(\mathrm{T}_{B})$, property $(\mathrm{F}_{B})$ in \cite{BFGM}; and property $(\mathrm{FF}_{B})$, property $(\mathrm{FF}_B)/\mathrm{T}$ in \cite{Mim}. Note that if $\Lambda_0 \trianglelefteq $$\Lambda$ is a \textit{normal} subgroup, then $B=$$B^{\rho(\Lambda_0)}\oplus B'_{\rho(\Lambda_0)}$ can be seen as a decomposition of $B$ as $\Lambda$-representation spaces.

\begin{defn}
Let $\Lambda$, $B$, $\rho$ be as in the setting of this section. 
\begin{itemize} 
   \item The representation $\rho$ of $\Lambda$ on $B$ is said to have \textit{almost invariant vectors}, written as $\rho \succ 1_{\Lambda}$, if the following holds: for any compact subset $K\subset \Lambda$, there exists a sequence of unit vectors $(\xi_n)_n$ in $B$ such that 
  $  \max_{s\in K} $$\|\xi_n -\rho(s)\xi_n \| \to 0$  as $n\to \infty$.
   \item A continuous map $b\colon \Lambda \to B$ is called a $\rho$\textit{-cocycle} if for any $g,h\in \Lambda$, 
   $b(gh)=$$b(g)+\rho(g)b(h)$ holds. The map $b$ is called a \textit{quasi-}$\rho$\textit{-cocycle} if $\sup_{g,h}$
   $\| b(gh)$$-b(g)-\rho (g)b(h)\| $$<\infty$ holds.
\end{itemize}
\end{defn}

\begin{defn}$($\cite{BFGM},\cite{Mim}$)$
Let $\Lambda$, $B$, $\rho$ be as in the setting of this section, and fix $B$. 
\begin{enumerate}[$($$1$$)$]
   \item A group pair $\Lambda \trianglerighteq \Lambda_0$ is said to have \textit{relative property }$(T_B)$ if the following holds true: for any $\rho$ of $\Gamma$ on $B$, the isometric linear representation $\rho'$, constructed by restricting $\rho$ on $B'_{\rho(\Lambda_0)}$, $\rho '\colon$$ \Lambda \to O(B'_{\rho(\Lambda_0)})$ does not have almost invariant vectors. The group $\Lambda$ is said to have \textit{property }$(T_B)$ if $\Lambda \trianglerighteq \Lambda$ has relative property $(\mathrm{T}_B)$. 
   \item A group pair $\Lambda \geqslant \Lambda_0$ is said to  have \textit{relative property }$(F_B)$ if for any $\rho$ of $\Lambda$ on $B$, every $\rho$-cocycle $b$ is a coboundary on the subgroup $\Lambda_0$, namely, there exists $\xi \in B$ such that for any $h\in \Lambda_0$, $b(h)=$$\xi -\rho (h)\xi$ holds. Equivalently, $b$ is bounded on $\Lambda_0$. The group $\Lambda$ is said to have \textit{property }$(F_B)$ if $\Lambda \geqslant \Lambda$ has relative property $(\mathrm{F}_B)$.
   \item A group pair $\Lambda \geqslant \Lambda_0$ is said to have \textit{relative property }$(FF_B)$ if for any $\rho$ of $\Lambda$ on $B$, every (continuous) quasi-$\rho$-cocycle $b$ is bounded on $\Lambda_0$. The group  $\Lambda$ is said to have \textit{property }$(FF_B)$ if $\Lambda \geqslant \Lambda$ has relative property $(\mathrm{FF}_B)$.
   \item The group $\Lambda$ is said to have \textit{property }$(FF_B)/T$    
   (, which means \textit{property }$(FF_B)$ \textit{modulo trivial part},) if for any $\rho$ of $\Lambda$ on $B$ and any (continuous) quasi-$\rho$-cocycle $b$, $b'(\Lambda)$ is bounded. Here $b'\colon \Lambda \to B'_{\rho(\Lambda)}$ means the projection of $b$ to a quasi-$\rho'$-cocycle which ranges in $B'_{\rho(\Lambda)}$, associated with the decomposition $B=B^{\rho (\Lambda)}\oplus B'_{\rho(\Lambda)}$. 
\end{enumerate}
If $B=\mathcal{B}$ is a family of Banach spaces, these seven properties are defined as having corresponding properties for all Banach spaces $E\in \mathcal{B}$.
\end{defn}
We make two remarks on the definition above: first, if $B$ is a general (possibly not superreflexive) Banach space, $B'_{\rho(\Lambda)}$ is replaced with $B/B^{\rho(\Lambda)}$ in the definitions of (relative) $(\mathrm{T}_B)$ and $(\mathrm{FF}_B)/\mathrm{T}$. Secondly, if $\Lambda$ has a compact abelianization, then property $(\mathrm{FF}_B)/\mathrm{T}$  implies property $(\mathrm{F}_B)$. In particular, the following holds true:

\begin{lem}\label{lem:abel}
For any $p\in (1,\infty)$, property $(\mathrm{FF}_{\mathcal{L}^p})/\mathrm{T}$ implies property $(\mathrm{F}_{\mathcal{L}^p})$. 
\end{lem}
\begin{proof}
Let $\Lambda$ have $(\mathrm{FF}_{\mathcal{L}^p})/\mathrm{T}$. Suppose $H:=\Lambda/\overline{[\Lambda,\Lambda]}$ is noncompact. Then consider the ($L^p$-)left regular representation $\lambda_H$ of $H$ in $L^p(H)$ and a representation $\rho$ of $\Lambda$ which is the pull-back of $\lambda_H$ by $\Lambda\twoheadrightarrow H$. Since $H$ is abelian and noncompact, $\lambda_H \succ 1_H$ but $\lambda_H \not\supset 1_H$. Therefore there exists a $\lambda_H$-cocycle which is not a coboundary (for details, see \cite[\S 3.a]{BFGM}). Thus we have a $\rho$-cocycle which is not a coboundary. However this contradicts $\rho \not\supset 1_{\Lambda}$ and $(\mathrm{FF}_{\mathcal{L}_p})/\mathrm{T}$ for $\Lambda$.
\end{proof}

By the Delorme--Guichardet theorem, if $\mathcal{B}$ is the family $\mathcal{H}$ of all Hilbert spaces $(=\mathcal{L}^2)$, then property $(\mathrm{F}_{\mathcal{H}})=(\mathrm{FH})$ is equivalent to property $(\mathrm{T}_{\mathcal{H}})=(\mathrm{T})$. We note that in the sprit of this, property $(\mathrm{FF}_{\mathcal{H}})$ was previously defined and called \textit{property }$(TT)$ by N. Monod \cite{Mon}. In the spirit of this, we also call $(\mathrm{FF}_{\mathcal{H}})/\mathrm{T}$ \textit{property} $(TT)/T$. By Lemma~\ref{lem:abel}, $(\mathrm{TT})/\mathrm{T}$ implies $(\mathrm{T})$.

We now explain one more concept which appears in Theorem~\ref{thm:criter}, namely, an \textit{ultraproduct} of Banach spaces. For precise definition and comprehensive treatment, see \cite{He}. Here we briefly recall the definition: Take a non-principal ultrafilter $\omega$ on $\mathbb{N}$ and fix it. For a sequence $((B_n, \|\cdot \|_n))_n$ of pairs of a Banach space and the norm, we define the \textit{ultraproduct} $\lim_{\omega}B_n=(B_{\omega}, \|\cdot \|_{\omega})$ as follows: we define $B_{\omega}$ as $( \bigoplus_{n\in \mathbb{N}} (B_n, \|\cdot \|_n) )_{\infty} $$/ \mathcal{N}$.
Here $( \bigoplus_{n\in \mathbb{N}} (B_n, \|\cdot \|_n) )_{\infty}$ means the set of sequences with bounded norms. We define a seminorm $\|\cdot \|$ there by $\|(\xi_n )_n\|:=$$\lim_{\omega}\| \xi_n\|_n$, and set $\mathcal{N}$ as the null subspace with respect to $\|\cdot \|$. Finally, define a norm $\|\cdot \|_{\omega}$ on $B_{\omega}$ as the induced norm by $\|\cdot \|$ above. We say a family $\mathcal{B}$ is \textit{stable under ultraproducts} if whenever $B_n\in \mathcal{B}$ for all $n\in \mathbb{N}$, $\lim_{\omega}B_n \in \mathcal{B}$ holds.

Before proving of Theorem~\ref{thm:criter}, we make a remark that ``condition $(i)$ implies the conclusion" has sprit of the original  Shalom's strategy \cite{Shal4} to prove Theorem~\ref{thm:SV}, and of an observation by M. Gromov. Precisely, for the proof of this part, one essential part is to use \textit{reduced}-1-cohomology. The original argument in   \cite{Shal2} uses conditionally negative definite functions, but Gromov \cite{Gr} observed that it can be done in terms of \textit{scaling limits} on a metric space, which are special cases of ultraproducts for Banach spaces. For details, see \cite[\S 5]{Mim} for instance.

\begin{proof}(\textit{Theorem~\ref{thm:criter}})
The proof follows from a combination of previous results in \cite{Mim}. With keeping the same notation as in Theorem~\ref{thm:criter}, we list necessary results: 
\begin{thm}$($\cite[Theorem 1.3, Theorem 6.4]{Mim}$)$\label{thm:res1}
Let $B$ be a superreflexive space or a family of them. Suppose $\mathrm{E}_2(A)\ltimes A^2$$\trianglerighteq A^2$ has relative property $(\mathrm{T}_B)$. Then for any $m\geq 3$, $\mathrm{SL}_m(A)\ltimes A^m$$\geqslant A^m$ has relative property $(\mathrm{F}_B)$. In fact, this pair has relative property $(\mathrm{FF}_B)$.
\end{thm}
\begin{thm}$($\cite[Theorem 5.5]{Mim}, Shalom's machinery$)$\label{thm:res2}
Let $\mathcal{B}$ be a family of superreflexive Banach spaces and $n\geq 3$. Suppose $\mathrm{SL}_n(A)\geqslant A^{n-1}$ has relative property $(\mathrm{F}_{\mathcal{B}})$, where $A^{n-1}$ sits on a unipotent part, from  $(1,n)$-th to $(n-1,n)$-th entries. If $\mathcal{B}$ is stable under ultraproducts, then $\mathrm{SL}_n(A)$ possesses property $(\mathrm{F}_{\mathcal{B}})$.
\end{thm}
\begin{prop}$($\cite[Proposition 6.6]{Mim}$)$\label{prop:res3}
Let $B$ be a superreflexive Banach space or a family of them, and let $n\geq 3$. Suppose $\mathrm{SL}_n(A)\geqslant A^{n-1}$ has relative property $(\mathrm{FF}_B)$, where $A^{n-1}$ sits in the same way as Theorem~$\ref{thm:res2}$. If $\mathrm{SL}_n(A)$ moreover satisfies property $(\mathrm{T}_B)$, then $\mathrm{SL}_n(A)$ possesses property $(\mathrm{FF}_B)/\mathrm{T}$.
\end{prop}
Indeed, for the last two of three, we need a quadruple $(\Lambda, \Lambda ',H_1,H_2)$ of a countable discrete group $\Lambda$ with finite abelianization and subgroups $\Lambda',H_1,H_2$ therein satisfying the following three conditions: 
\begin{enumerate}[$($$1$$)$]
   \item The group $\Lambda$ is generated by $H_1$ and $H_2$ together.
   \item The subgroup $\Lambda'$ normalizes $H_1$ and $H_2$.
   \item The group $\Lambda$ is boundedly generated by $\Lambda'$, $H_1$, and $H_2$.
\end{enumerate}
Here we say a subset $\mathcal{S} \subset \Lambda$ containing the unit $e\in \Lambda$ \textit{boundedly generates} a group $\Lambda$ if there exists $N\in \mathbb{Z}_{>0}$ such that $\mathcal{S}^N=\Lambda$ holds (this equality means, any $g\in \Lambda$ can be expressed as a product of $N$ elements in $\mathcal{S}$). We warn that in some other literature, the terminology \textit{bounded generation} is used only for the following \textit{confined} case: $\mathcal{S}$ is a finite union of cyclic subgroups of $\Lambda$. These properties relate to some forms of \textit{Shalom properties}, which are used in the proofs of Theorem~\ref{thm:res2} and Proposition~\ref{prop:res3}. For more details, compare \cite[Definition~5.4, Definition~6.5]{Mim}.

We get these last two results stated above by letting $(\Lambda, \Lambda',H_1,H_2)$ in the original statements be $(\mathrm{SL}_n(A), \mathrm{SL}_{n-1}(A),A^{n-1},A^{n-1})$. Here $\Lambda'$ sits in the left upper corner of $\mathrm{SL}_n(A)$; $H_1\cong A^{n-1}$ sits in $\mathrm{SL}_n(A)$ as a unipotent part, from  $(1,n)$-th to $(n-1,n)$-th entries; and $H_2 \cong A^{n-1}$ sits in $\mathrm{SL}_n(A)$ as  a unipotent part, from  $(n,1)$-th to $(n,n-1)$-th entries. Condition $(3)$ for this case follows from a deep theorem of Vaserstein \cite{Vas}.

First, we deal with the case of that condition $(i)$ is fulfilled. Then by combining Theorem~\ref{thm:res1} and Theorem~\ref{thm:res2}, we obtain that $\mathrm{SL}_{n\geq 4}(A)$ has property $(\mathrm{F}_{\mathcal{B}})$ (note that $\mathrm{SL}_m(A)\ltimes A^m$ naturally injects into $\mathrm{SL}_{m+1}(A)$). Since property $(\mathrm{F}_{\mathcal{B}})$ implies property $(\mathrm{T}_{\mathcal{B}})$ (it is a general fact. See \cite[Theorem 1.3]{BFGM}), $\mathrm{SL}_{n\geq 4}(A)$ has property $(\mathrm{T}_{\mathcal{B}})$. Then by Proposition~\ref{prop:res3}, we obtain property $(\mathrm{FF}_{\mathcal{B}})/\mathrm{T}$ as well.

Secondly, we consider the case of that condition $(ii)$ is satisfied. In general, we may not apply Theorem~\ref{thm:res2}. However in this case, we can first apply Theorem~\ref{thm:res1}, and next appeal directly to Proposition~\ref{prop:res3}. The point here is that since relative $(\mathrm{T})$ implies relative $(\mathrm{T}_{\mathcal{B}})$, the pairs $\mathrm{E}_2(A)\ltimes A^2$$\trianglerighteq A^2$; and $\mathrm{SL}_{n}(A)$$\trianglerighteq \mathrm{SL}_{n}(A)$ ($n\geq 3$) have relative $(\mathrm{T}_{\mathcal{B}})$ (these respectively follow from \cite{Shal1} and Theorem~\ref{thm:SV}). Thus for $\mathrm{SL}_{n\geq 4}(A)$, we obtain property $(\mathrm{FF}_{\mathcal{B}})/\mathrm{T}$. Since $\mathrm{SL}_{n\geq 4}(A)$ has the trivial abelinanization, we get $(\mathrm{F}_{\mathcal{B}})$ as well.
\end{proof}

We mention Bader--Furman--Galander--Monod \cite[\S 4.a]{BFGM}) have shown that relative $(\mathrm{T})$ implies relative $(\mathrm{T}_{L^p(\mu)})$ for any $\sigma$-finite measure $\mu$. Moreover, Heinrich \cite{He} has proven the family $\mathcal{L}^p$ satisfies condition $(i)$. Thus although we saw property $(\mathrm{F}_{\mathcal{L}^p})$ is stronger than property $(\mathrm{T})$ for $p$ large enough (recall this from the introduction), for universal lattices with $n\geq 4$ we even obtain property $(\mathrm{FF}_{\mathcal{L}^p})/\mathrm{T}$  (\cite[Theorem 1.5]{Mim}). Also in \cite[Remark 6.7]{Mim}, we have obtained $(\mathrm{TT})/\mathrm{T}$ (=$(\mathrm{FF}_{\mathcal{L}^2})/\mathrm{T}$) for $\mathrm{SL}_{n\geq 3}(A)$. Note that $(\mathrm{TT})/\mathrm{T}$ is strictly stronger than $(\mathrm{T})$ because any non-elementary hyperbolic group, including one with $(\mathrm{T})$, is known to admit an unbounded quasi-cocycle into the left regular representation, see \cite{MMS}.

In Section~\ref{sec:Cp}, we shall see the family $\mathcal{C}_p$ satisfies condition $(ii)$. On a family $\mathcal{B}$ having condition $(i)$ but not satisfying condition $(ii)$, see Section~\ref{sec:concl}. 

\section{Relative property $(\mathrm{T})$ implies relative property $(\mathrm{T}_{\mathcal{C}_p})$}\label{sec:Cp}
We refer to \cite{PX} for comprehensive treatments on noncommutative $L^p$-spaces including $C_p$. First of all, the Clarkson-type inequality (for instance, see\cite[\S 5]{PX}) implies that $C_p$ for any $p(\in (1,\infty))$ is uniformly convex and uniformly smooth. 

Before proceeding to the proof of Proposition~\ref{prop:C_p}, we shortly recall the strategy in \cite[\S 2.e, \S 4.a]{BFGM} of proving the following commutative version of this theorem: property $(\mathrm{T})$ implies  property $(\mathrm{T}_{L^p(\mu)})$ for $\sigma$-finite measure space $\mu$. The keys to their proof are the following two tools:
\begin{enumerate}[Tool $1$.]
   \item  (The \textit{Mazur map}: interpolation between $L^p$-spaces, see \cite{BL}) For $p,r\in (1,\infty)$ and $\sigma$-finite measure, the map $M_{p,r}$$\colon L^p(\mu) \to L^r(\mu)$; $ M_{p,r}(f)=\mathrm{sign}(f)\cdot |f|^{p/r}$
   is a (non-linear) map, and this induces a uniformly continuous homeomorphism between the unit spheres $M_{p,r}\colon $$S(L^{p}(\mu))\to S(L^{r}(\mu))$.
   \item  (The \textit{Banach--Lamperti theorem}: classification of linear isometries on an $L^p$-space, see \cite{FJ}) For any $1<p<\infty$ with $p\ne 2$, any linear isometry $V$ of $L^p(X,\mu)$ has the form 
   $$
   Vf(x)=f(F(x))h(x)\left( \frac{dF_{*}\mu}{d\mu}(x)\right) ^{\frac{1}{p}}, 
   $$
   where $F$ is a measurable, measure class preserving map of a Borel space $(X,\mu)$, and $h$ is a measurable function with $|h(x)|=1$ almost everywhere.
\end{enumerate}

Their proof goes as follows: suppose a group $\Lambda$ does not have property $(\mathrm{T}_{L^p(\mu)})$. Then there exists a (continuous) isometric linear representation $\rho$ on $B=L^p(\mu)$ such that $\rho ' \succ 1_{\Lambda}$ (, namely, $\rho '$ has almost invariant vectors). Here $\rho '$ is the restriction of $\rho$ on the subspace $B':=B'_{\rho (\Lambda)}$, recall the definitions above from Section~\ref{sec:criter}. Through Tool $1$, define $\pi$ by $\pi (g)=M_{p,2}\circ \rho (g) \circ M_{2,p}$ $(g\in \Lambda)$. Then thanks to Tool 2, one can show this $\pi$ maps each $g\in \Lambda$ to a \textit{linear} (unitary) operator on the Hilbert space $\mathfrak{H}:=L^2(\mu)$. Thus one obtains the unitary representation $\pi \colon$$\Lambda \to \mathcal{U}(\mathfrak{H})$. Finally, by uniform continuity of the Mazur maps, it is not difficult to see that $\rho ' \succ 1_{\Lambda}$ implies $\pi ' \succ 1_{\Lambda}$ (, where $\pi '$ is the restriction of $\pi$ on the orthogonal complement of $\mathfrak{H}^{\pi (\Lambda)}$). This means $\Lambda$ does not have property $(\mathrm{T})$. Hence property $(\mathrm{T})$ implies property $(\mathrm{T}_{L^p(\mu)})$. Also in this argument one can lead the same conclusion for relative properties. 

\begin{proof}(\textit{Proposition~\ref{prop:C_p}})
Puschnigg \cite{Pu} shown a noncommutative  analogue of Tool 1, for the case of $p$-Schatten class $C_p$: 
\begin{thm}$($Puschnigg \cite[Corollary 5.6]{Pu}$)$\label{thm:pu}
Let $1<p,r<\infty$. Then the noncommutative Mazur map
$M_{p,r}\colon S(C_p)\to S(C_r);$ $ a\mapsto u\cdot |a|^{p/r}$
is a uniform continuous homeomorphims between unit spheres. Here $a=u\cdot |a|$ is a polar decomposition. 
\end{thm}
In \cite{Pu}, he considered isometric linear representations on $C_p$ which come from unitary representations, and hence he did not need an analogue of Tool 2. For our case, it is needed, and is obtained by J. Arazy \cite{Ara}. Since our Hilbert space $\mathfrak{H}$ is separable, by choosing a countable orthonormal basis, we can identify $\mathfrak{H}$ with a square integrable sequence space $\ell^2$. Through this identification, we can consider the \textit{transpose map}; $ a\mapsto {}^{T}a$ on $\mathbb{B}(\mathfrak{H})\cong \mathbb{B}(\ell^2)$. Although the transpose map depends on the choices of bases, it is easy to obtain the following facts:

\begin{lem}\label{lem:trans}
Stick to the setting in above.
\begin{enumerate}[$(i)$]
  \item The transpose map is a linear isometry on each $C_p$.
  \item The transpose map is compatible with the adjoint operation. Namely, for any $a\in \mathbb{B}(\ell^2)$, 
    $
    ({}^Ta)^{*} ={}^T(a^{*})
    $.
  \item For any $a,b\in \mathbb{B}(\ell^2)$, ${}^T(ab)={}^Tb{}^Ta$.
  \item If $u$ is a unitary, then so is ${}^T u$.
  \item If $t$ is positive and $\alpha>0$, then ${}^T t$ is positive and ${}^T(t^{\alpha})=({}^Tt)^{\alpha}$.
\end{enumerate}
\end{lem}
Now we state the theorem of Arazy:
\begin{thm}$($Arazy \cite{Ara}$)$\label{thm:Ara}
Let $1<p<\infty$ with $p\ne 2$ and $C_p$ be the space of $p$-Schatten class operators on $\ell^2$. Then every linear isometry $V$ on $C_p$ is either of the following two forms:
  \begin{enumerate}[$($$1$$)$]
    \item there exist unitaries $w,v\in \mathcal{U}(\ell^2)$ such that $V\colon$ $a\mapsto w a v$;
    \item there exist unitaries $w,v\in \mathcal{U}(\ell^2)$ such that $V\colon$ $a\mapsto w  {}^T av$.
\end{enumerate}
\end{thm}

Thanks to the two theorems above, by following footsteps of Bader--Furman--Gelander--Monod we accomplish the conclusion. Indeed, for linearity of the composition, let us take a linear isometry $V$ on $C_p$. By Lemma~\ref{lem:trans}, for $\widetilde{V}=M_{p,2}\circ V\circ M_{2,p} $$\colon S(C_2)\to S(C_2)$ we have the following (recall that one can take a partial isometry in polar decomposition as a unitary):
\begin{enumerate}[$(i)$]
   \item in the case of $(1)$, for any $x\in C_2$ with a polar decomposition $x=u|x|$, a polar decomposition of $w u|x|^{2/p}v$ is $(wuv)(v^{*}|x|^{2/p}v)$. Therefore, we have 
   \begin{align*}
   \widetilde{V}\cdot x&= wuv (v^{*}|x|^{2/p}v)^{p/2} =wuv v^{*}|x|v \\
                       &= wu|x|v=wxv.
   \end{align*}
   \item in the case of $(2)$, for any $x\in C_2$ with a polar decomposition $x=u|x|$, a polar decomposition of $w {}^T(u|x|^{2/p})v$$= w ({}^T|x|)^{2/p}({}^Tu)v$ is 
   $$
   (w{}^Tuv)\{v^{*}({}^Tu)^{*}({}^T|x|)^{2/p}({}^T u)v\}. 
   $$
   Therefore, we have 
   \begin{align*}
   \widetilde{V}\cdot x&= w {}^T u  v \{v^{*}({}^Tu)^{*}({}^T|t|)^{2/p}({}^Tu)v\}^{p/2} \\
                       &=w {}^Tu v v^{*}({}^Tu)^{*}({}^T|x|){}^Tuv \\
                       &= w{}^T|x|{}^Tuv =w {}^T(u|x|)v=w({}^Tx)v.
   \end{align*}
\end{enumerate}

Hence in both cases $\widetilde{V}$ is linear. 
Now recall that $C_2$, the space of Hilbert--Schmidt operators, is indeed a Hilbert space equipped with the inner product: $\langle x\mid y\rangle :=\mathrm{Tr}(y^{*}x)$ (this holds for noncommutative $L^2$ spaces in general setting). 
\end{proof}

Proposition~\ref{prop:C_p} together with Theorem~\ref{thm:criter} immediately implies Theorem~\ref{thm:Cp} for universal lattices and group quotients of them. In Section~\ref{sec:fi}, we deal with the case of finite index subgroups by utilizing $p$-inductions.

\section{\textbf{Finite index subgroups, Lattices, and $L^p$-inductions}}\label{sec:fi}
The $p$-induction of (quasi-)cocycles is given by Shalom \cite[\S 3. III]{Shal3} for $p=2$, and later by Bader--Furman--Gelander--Monod \cite[\S 8]{BFGM} for general $p$. 

Let $G$ be a (locally compact second countable) group, $\Gamma\leqslant G$ 
be a lattice, and $\mathcal{D}$ be a Borel fundamental domain for $\Gamma$ (namely, $\mathcal{D}$ is a Borel subset of $G$ such that $G=\bigsqcup_{\gamma\in \Gamma}\mathcal{D}\gamma$). For the existence of such a domain, see \cite[\S B]{BHV}. We let $\mu$ be a Haar measure of $G$ with $\mu(\mathcal{D})=1$; and we identify $\mathcal{D}$ with $G/\Gamma$ and regard $\mathcal{D}$ as a (left) $G$-space. We define a Borel cocycle $\beta \colon G\times \mathcal{D} \to \Gamma$ by the following rule:
$$
\beta (g,d)=\gamma \textrm{ if and only if }g^{-1}d\gamma \in \mathcal{D}.
$$
Now for given isometric $\Gamma$-representation $\sigma $ on a Banach space $B$ and a (quasi-)$\sigma$-cocycle $b\colon \Gamma \to B$, under some integrability condition in below, we can define the $p$\textit{-induction} $\tilde{b}$ of $b$ as follows:
\begin{itemize}
  \item the $p$-induced Banach space is $L^p(\mathcal{D},B)$, the space of $p$-Bochner integrable functions;
  \item the $p$-induced representation $\rho=\mathrm{Ind}_{\Gamma}^G \sigma$ in $L^p(\mathcal{D},B)$ is defined as follows: for $g\in G$, $\xi \in L^p(\mathcal{D},B)$ and $d\in \mathcal{D}$, 
  $$
  \rho (g)\xi(d):= \sigma (\beta (g,d)) \xi (g^{-1}\cdot d);
  $$
  \item we define the $\tilde{b}\colon G\to L^p(\mathcal{D},B)$ as 
  $$
   \tilde{b}(g)(d):=b(\beta (g,d)) \ \ (g\in G,\ d\in \mathcal{D}), 
  $$
  \textit{provided that $\tilde{b}(g)$ is $p$-integrable for all $g\in G$.}
\end{itemize}
Then under the condition above, this $\tilde{b}$ becomes a (quasi-)$\mathrm{Ind}_{\Gamma}^G \sigma$-cocycle.

\begin{rem}\label{rem:int}
The $p$-induction procedure of (quasi-)cocycles requires the integrable condition above, and in general it is a subtle problem to determine whether this holds. However, it is known that this condition is satisfied for all $p$ in the following two cases:
\begin{enumerate}[$(1)$]
  \item if $\Gamma$ is cocompact in $G$, in particular, if $\Gamma$ is a finite index subgroup of (a countable) $G$, then for any $\mathcal{D}$ and any $b$ the integrability condition is fulfilled;
  \item if $G$ is a semisimple algebraic group with each simple factor having a local rank at most $2$, then there exists $\mathcal{D}$ such that for any $b$ the integrability condition is fulfilled.
\end{enumerate}
For case $(1)$, it is almost trivial to check the integrability. However for case $(2)$, a proof of this fact is considerably involved: one needs a deep result on length functions on higher rank lattices \cite{LMR}, see \cite[\S 2]{Shal2} how to deduce the integrability.
\end{rem}

Now we state the strategy to prove Theorem~\ref{thm:Cp} for finite index subgroups of universal lattices; and Theorem~\ref{thm:highrank}. At the beginning, we consider a general setting. Let $G$ be a (locally compact and second countable) group and $\Gamma$ be a lattice in $G$. Let $\sigma$ be an isometric $\Gamma$-representation in $C_p$. Take a Borel fundamental domain $(\mathcal{D},\mu)$ of $\Gamma$ and consider the $p$-induction (with the same $p$ as $C_p$) $\rho= \mathrm{Ind}_{\Gamma}^G \sigma$ of $\sigma$ on $E=L^p(\mathcal{D},C_p)$. Note that $E$ is identical to the following space:
$$
\{ x\colon \mathcal{D}\to C_p: \| x\|_p^p:=\int_{d\in \mathcal{D}}\mathrm{Tr}(|x(d)|^{p}) d\mu <\infty\}.
$$

We need to have analogues of Tool $1$ and Tool $2$ in Section~\ref{sec:Cp}. Note that $E=L^p(\mathcal{D},C_p)$ can be seen as the noncommutative $L^p$-space associated with the von Neumann algebra $L^{\infty}(\mathcal{D},\mathbb{B}(\mathfrak{H}))$$\cong L^{\infty}(\mathcal{D})\otimes \mathbb{B}(\mathfrak{H})$ (acting on $L^{2}(\mathcal{D})\otimes \mathfrak{H}$) with the canonical (normal faithful) semifinite trace $\tau= \tau_{L^{\infty}} \otimes \mathrm{Tr}$: for $x\in L^{\infty}(\mathcal{D},\mathbb{B}(\mathfrak{H})^+)$, 
$$
\tau (x):= \int_{d\in \mathcal{D}} \mathrm{Tr}(x(d))  d\mu \in [0,\infty].
$$
For the definition of noncommutative $L^p$-space $L_p(\mathcal{A},\tau)$ associated with a semifinite von Neumann algebra $(\mathcal{A},\tau)$, see \cite{PX}. We use this symbol  to distinguish it from $L^p(\mathcal{D},C_p)$ above. 
For Tool $1$, we utilize the following result of Puschnigg: 
\begin{thm}$($Puschnigg, \cite[Corollary 5.7]{Pu}$)$\label{thm:pusc}
For $1<p,r<\infty$, the analogue of the Mazur map $($Theorem~$\ref{thm:pu})$
$$
\tilde{M}_{p,r}\colon L^p(\mathcal{D},C_p) \to L^r(\mathcal{D},C_r); 
x \mapsto u \cdot |x|^{p/r}
$$
is  a uniform continuous homeomorhism on unit spehres. Here $x=u|x|$ is a polar decomposition. 
\end{thm}
Note that the unifrom convexity and the uniform smoothness of $E$ follow from a general theory, see \cite{FiPi}. 

For Tool $2$, F. J. Yeadon has shown the following theorem on isometries in the noncommutative $L^p$-space associated with a semifinite von Neumann algebra: 
\begin{thm}$($Yeadon, \cite[Theorem2]{Yea}$)$\label{thm:ye}
Let $\mathcal{A}$ be a von Neumann algebra with a faithful semifinite normal trace $\tau$ and let $L_p(\mathcal{A},\tau)$ be the associated noncommutative $L^p$-space with $p\ne 2$ in $(1,\infty)$. Suppose $V$ is a linear isometry. Then there exist, uniquely, a partial isometry $w\in \mathcal{A}$,  a normal Jordan $*$-monomorphism, and  an $($unbounded$)$ positive self-adjoint operator $N$ affiliated with $\mathcal{A}\cap J(\mathcal{A})'$ such that 
\begin{itemize}
   \item $w^*w=J(1)=\textrm{the support projection of } N$;
   \item for any $x\in {\mathcal{A}}^{+}$, $\tau (x)=\tau (N^pJ(x))$;
   \item for any $x\in \mathcal{A}\cap L_p(\mathcal{A},\tau )$, $V\cdot x=wNJ(x)$.\end{itemize}
\end{thm}
Here $\mathcal{A}'$ denotes the commutant of $\mathcal{A}$, and a closed and densely defined operator $L$ is said to be \textit{affiliated with} a von Neumann algebra $\mathcal{A}_0$ (both acting on the same Hilbert space) if every unitary $u$ in $\mathcal{A}_0'$ carries the domain of $N$, onto itself and satisfies $uNu^{*}=N$ there. A \textit{Jordan $*$-monomorphism} of a von Neumann algebra $\mathcal{A}$ is an injective linear map $J\colon \mathcal{A}\to \mathcal{A}$ which satisfies for any $x\in \mathcal{A}$, $J(x^2)=J(x)^2$. 

With the aid this theorem, we claim the following:
\begin{prop}\label{prop:LpCp}
For a finite measure space $\mathcal{D}$ and $p\ne 2$ in $(1,\infty)$, let $V$ be a linear isometry on $L^p(\mathcal{D},C_p)$. Then the following composition
$$
\tilde{V}= \tilde{M}_{p,2}\circ V \circ \tilde{M}_{2,p}
$$
is a $\mathrm{linear}$ isometry map on $L^2(\mathcal{D},C_2)$.
\end{prop}
Note that $L^2(\mathcal{D},C_2)$ is a Hilbert space (compare with Section~\ref{sec:Cp}).

\begin{proof}(\textit{Proposition~\ref{prop:LpCp}})
For the proof, we need the following theorem of St\o mer on structures of Jordan monomorphisms (note that $e$ and $f$ below are orthogonal: $ef=fe=0$):
\begin{thm}$($St\o mer, \cite[Lemma 3.2]{St}$)$\label{thm:sto}
Let $\mathcal{A}$, $\mathcal{M}$ be von Neumann algebras and $J$ be a normal Jordan $*$-monomorphism from $\mathcal{A}$ into $\mathcal{M}$ such that the von Neumann algebra generated by $J(\mathcal{A})$ equals $\mathcal{M}$. Then there exist two central projections $e$ and $f$ in $\mathcal{M}$ with $e+f=1_{\mathcal{M}}$ such that the map $J_1\colon x \mapsto J(x)e$ is a $*$-homomorphism and that the map $J_2\colon x \mapsto J(x)f$ is an anti-$*$-homomorphism.
\end{thm}
We apply Theorem~\ref{thm:ye} and Theorem~\ref{thm:sto} to our case: $\mathcal{A}=L^{\infty}(\mathcal{D})\otimes \mathbb{B}(\mathfrak{H})$ with the canonical trace $\tau$. Thus for $V$, we obtain $w$, $J$, and $N$ in the statement of Theorem~\ref{thm:ye} and let $q$ be the projection $J(1_{\mathcal{A}})\in \mathcal{A}$. Set $\mathcal{M}$ be the von Neumann algebra generated by $J(\mathcal{A})$ with putting $1_{\mathcal{M}}=q$, and through Theorem~\ref{thm:sto} we obtain $e,f$ central in $\mathcal{M}$. Take an arbitrary element $x$ in $L_{2}(\mathcal{A},\tau) \cap \mathcal{A}$, namely, $x$ is an element in $L^2(\mathcal{D}, C_2)$ which satisfies $
\mathrm{ess.\ sup}_{d\in \mathcal{D}} \|x(d)\| <\infty$. 
Let $x=u|x|$ be a polar decomposition of $x$ with $u \in \mathcal{A}$ being a unitary. Then 
\begin{align*}
(V\circ \tilde{M}_{2,p}) (x)= V\cdot (u|x|^{2/p})&=wNJ(u|x|^{2/p})e+wNJ(u|x|^{2/p})f \\
&= wNJ(u)J(|x|)^{2/p}e+wNJ(|x|)^{2/p}J(u)f.
\end{align*}
Hence the following is a polar decomposition of $(V\circ \tilde{M}_{2,p}) (x) $ with $wJ(u)$ being a partial isometry:
\begin{align*}
(V\circ \tilde{M}_{2,p}) (x)&=(wJ(u)e+wJ(u)f)(NJ(|x|)^{2/p}e+J(u)^{*}NJ(|x|)^{2/p}J(u)f) \\
&=(wJ(u))(NJ(|x|)^{2/p}e+J(u)^{*}NJ(|x|)^{2/p}J(u)f)
\end{align*}
Here we use the facts that $J(u)$ is a partial isometry with $J(u)^*J(u)=J(u)J(u)^*=q$; $e,f$ are projections with $e+f=q$ which are central in $\mathcal{M}$; and that $N$ is affiliated with $\mathcal{A}\cap \mathcal{M}'=\mathcal{A}\cap J(\mathcal{A})'$. Therefore, one has the following:
\begin{align*}
(\tilde{M}_{p,2}\circ V\circ \tilde{M}_{2,p}) (x)&=(wJ(u))(NJ(|x|)^{2/p}e+J(u)^*NJ(|x|)^{2/p}J(u)f)^{p/2} \\
&= (wJ(u)e+wJ(u)f)(N^{p/2}J(|x|)e+J(u)^{*}N^{p/2}J(|x|)J(u)f) \\
&= w N^{p/2}J(u)J(|x|)e + w N^{p/2}J(|x|)J(u)f \\
&=w N^{p/2}J(u|x|)e + w N^{p/2}J(u|x|)f = w N^{p/2}J(u|x|) \\
&=w N^{p/2}J(x).
\end{align*}
This means that $\tilde{M}_{p,2}\circ V\circ \tilde{M}_{2,p}$ is a linear $2$-isomery at least from $L_{2}(\mathcal{A},\tau) \cap \mathcal{A}$ to $L_{2}(\mathcal{A},\tau)$. Since in our case  $L_{2}(\mathcal{A},\tau) \cap \mathcal{A}$ is ($2$-)dense in $L_{2}(\mathcal{A},\tau)$, this map extends to a linear isometry on $L_{2}(\mathcal{A},\tau)$, as desired. 
\end{proof}
Theorem~\ref{thm:pusc} together with Proposition~\ref{prop:LpCp} lead us to the following corollary (compare with Section~\ref{sec:Cp}):
\begin{cor}\label{cor:relLp}
Let $\mathcal{D}$ be a finite measure space and $p\ne 2$ in $(1,\infty)$. Then among locally compact and second countable groups, relative property $(\mathrm{T})$ implies relative property $(\mathrm{T}_{L^p(\mathcal{D},C_p)})$.
\end{cor}

\begin{proof}(\textit{Theorem~\ref{thm:Cp}} for finite index subgroups; and \textit{Theorem~\ref{thm:highrank}})

Frist we prove Theorem~\ref{thm:highrank}. Recall the argument in \cite[\S 5]{BFGM} of deducing $(\mathrm{F}_{\mathcal{L}^p})$ from (strong) relative $(\mathrm{T}_{\mathcal{L}^p})$ for higher rank algebraic groups (they utilizes the generalized Howe--Moore property, see \cite[\S 9]{BFGM}). Together with Proposition~\ref{prop:C_p} and Corollary~\ref{cor:relLp}, this argument implies $(\mathrm{F}_{\mathcal{C}_p})$ and $(\mathrm{F}_{L^p(\mathcal{D},C_p)})$ for higher rank algebraic groups. By $p$-induction process (it is possible, see $(2)$ of Remark~\ref{rem:int}), the later property implies $(\mathrm{F}_{\mathcal{C}_p})$ for lattices therein, see \cite[\S 8]{BFGM}.

Finally, we end the proof of Theorem~\ref{thm:Cp} by dealing with finite index subgroups of universal lattices. Let $G=\mathrm{SL}_{n\geq 4}(A)$ be the universal lattice and $\Gamma \leqslant G$ be a finite index subgroup. Suppose that $\Gamma$ does not have $(\mathrm{FF}_{\mathcal{C}_p})/\mathrm{T}$. Then there exists an isometric $\Gamma$-representation $\sigma$ on $C_p$ and a  quasi-$\sigma$-cocycle $b$ such that $b'\colon \Gamma \to (C_p)_{\sigma (\Gamma)}'$ is unbounded. Take a $p$-induction of $b$ (it is possible in view of Remark~\ref{rem:int} $(1)$), and get the induced representation $\rho=\mathrm{Ind}_{\Gamma}^G\sigma $ and the induced quasi-$\rho$-cocycle $c=\tilde{b}$ in $E={\ell}^p(\mathcal{D},C_p)$. 

Now observe that Theorem~\ref{thm:criter} and Corollary~\ref{cor:relLp} prove that $G$ has $(\mathrm{FF}_E)/\mathrm{T}$. This in particular implies the map
 $
c'\colon G \to E'_{\rho(G)}
$
 must be bounded. On the other hand, by the construction of $c$, the restriction $c'\mid_{\Gamma}$ can be identified with the quasi-$\sigma'$-cocycle $b'\colon \Gamma \to (C_{p})_{\sigma (\Gamma)}'$ and hence is unbounded. Since the measure space $\mathcal{D}$ consists of atoms, this forces $c'$ to be unbounded. It is a contradiction. Therefore $\Gamma$ must have $(\mathrm{FF}_{\mathcal{C}_p})/\mathrm{T}$, and also have $(\mathrm{F}_{\mathcal{C}_p})$ because $\Gamma$ has finite abelianization (for instance, this follows from Theorem~\ref{thm:SV}).
\end{proof}

\begin{rem}
Induction of quasi-cocycles from a lattice to a locally compact (topological) group is delicate in general. This is because for non-atomic finite measure space, the unboundedness of some value does not necessarily imply the unboundedness of $L^p$-norm. However, Burger and Monod have overcame this difficulity by showing the induced map $
H^2_{\mathrm{b}}(\Gamma;B,\sigma) \to H^2_{\mathrm{cb}}(G;L^p(\mathcal{D},B),\mathrm{Ind}_{\Gamma}^G\sigma)
$ 
is injective (here $H^{\bullet}_{\mathrm{cb}}$ denotes the continuous bounded cohomology). For details, see \cite{BM1}, \cite{BM2}.
\end{rem}

\section{\textbf{Concluding remarks}}\label{sec:concl}
\begin{rem}\label{rem:exa}
We will explain an example of a family $\mathcal{B}$ of Banach spaces which satisfies condition $(i)$ but fails to fulfill condition $(ii)$ in Theorem~\ref{thm:criter}. This example is employed for establishing property $(\mathrm{F}_{[\mathcal{H}]})$ for universal lattices in \cite{Mim}, where $[\mathcal{H}]$ denotes the family of Banach spaces $Y$ which is isomorphic to a Hilbert space; namely, on which there exists a Hilbert norm $\| \cdot \|_{\mathrm{Hilb}}$ such that there exists $C\geq 1$ such that $C^{-1}\| \cdot \|_{\mathrm{Hilb}}$$\leq \| \cdot \|_{Y}$$\leq C\| \cdot \|_{\mathrm{Hilb}}$ (the infimum of such $C$ is called the \textit{norm ratio}). The family $[\mathcal{H}]$ is not stable, but for any $M\geq 1$, the family $[\mathcal{H}]_M$ of any $Y$ isomorphic to a Hilbert space with norm ratio $\leq M$, is stable under ultraproducts. 
  Theorem 1.3 in \cite{Mim} states that $\mathrm{E}_2(A) \ltimes A^2 $$\trianglerighteq A^2$ has relative property $(\mathrm{T}_{[\mathcal{H}]})$, and the family $[\mathcal{H}]_M$ satisfies condition $(i)$ of Theorem~\ref{thm:criter}. Therefore $\mathrm{SL}_{n\geq 4}(A)$ has property $(\mathrm{F}_{[\mathcal{H}]_M})$ for any $M$. Thus property $(\mathrm{F}_{[\mathcal{H}]})$ for $\mathrm{SL}_{n\geq 4}(A)$ is also verified. 
 
 Note that, by considering $\|\cdot \|_{\mathrm{Hilb}}$, one can interpret property $(\mathrm{F}_{[\mathcal{H}]})$ as the fixed point property with respect to all affine \textit{uniformly bi-Lipschitz} action on Hilbert spaces (similarly, property $(\mathrm{T}_{[\mathcal{H}]})$ is interpreted as a property in terms of \textit{uniformly bounded} linear representations). We note that in \cite{BFGM}, property $(\mathrm{F}_{[\mathcal{H}]})$ and $(\mathrm{T}_{[\mathcal{H}]})$ are respectively called property $(\overline{F}_{\mathcal{H}})$ and $(\overline{T}_{\mathcal{H}})$. 
 
We also mention that $\mathcal{B}=[\mathcal{H}]$ and $\mathcal{B}=[\mathcal{H}]_M$ for sufficiently large $M$ do \textit{not} satisfy condition $(ii)$ in Theorem~\ref{thm:criter}. This is due to an unpublished result of Shalom that $\mathrm{Sp}_{n,1}$ and lattices therein fails to have property $(\mathrm{T}_{[\mathcal{H}]})$. 
\end{rem}

\begin{rem}\label{rem:typeII}
In the view of Theorem~\ref{thm:ye},  our fixed point theorems (and $(\mathrm{FF})/\mathrm{T}$-type property) can be extended, with some effort, to the cases of noncommutative $L^p$-spaces associated with semifinite von Neumann algebras. For uniform continuity of the associated Mazur maps on unit spehres, one can utilize an inequality of Kosaki \cite[Proposition 7]{Kos}. Compare with the proof of \cite[Corollary 5.6]{Pu}.

We mention that even from type $III$ von Neumann algebras (, namely, beyond semifinite cases), one can construct the associated noncommutative $L^p$-spaces. For details, see \cite{PX}. For classification of linear isometries on a general noncommutative $L^p$-space, see \cite{JRS} and \cite{She}.

Finally, we warn that the class $\mathcal{C}_p$ is \textit{not} stable under ultraproducts. To have stability, we need to consider noncommutative $L^p$-spaces associated with type $III$ von Neumann algebras as well. For details, see \cite{Ray}. 
\end{rem}

\begin{rem}
It is a problem of high interest to determine whether higher rank lattices and universal lattices have property $(\mathrm{F}_{\mathcal{B}_{\mathrm{uc}}})$ for $\mathcal{B}_{\mathrm{uc}}$ being the family of all uniformly convex Banach spaces. One of the main motivations for studying this problem is this relates to uniform (non-)embeddability of expander graphs, and that relates the coarse geometric Novikov conjecture \cite{KY} and (a possible direction to construct a counterexample of the surjectivity-side of) the Baum--Connes conjecture \cite{HLS}. There is a breakthrough by V. Lafforgue \cite{Laf1}, \cite{Laf2}, and his results imply $\mathrm{SL}_{n\geq 3}(F)$ ($F$ is a non-archimedean local field) and cocompact lattices therein have $(\mathrm{F}_{\mathcal{B}_{\mathrm{uc}}})$. For archimedean local field cases or noncocompact lattice cases, it does not seem any result is known for this problem.

In \cite{BFGM}, Bader--Furman--Gelander--Monod observed that for higher rank groups and lattices, in order to verify property $(\mathrm{F}_{\mathcal{B}_{\mathrm{uc}}})$ it suffices to show that $\mathrm{SL}_2(\mathbb{R}) \ltimes \mathbb{R}^2 $$\trianglerighteq \mathbb{R}^2$ and symplectic version of this pair have relative property $(\mathrm{T}_{\mathcal{B}_{\mathrm{uc}}})$. Thanks to Theorem~\ref{thm:criter}, we have an analogue of this observation for universal lattices; namely, the following:``\textit{if }$\mathrm{E}_2(A) \ltimes A^2 $$\trianglerighteq A^2$\textit{ has relative property }$(\mathrm{T}_{\mathcal{B}_{\mathrm{uc}}})$\textit{, then }$\mathrm{SL}_{n\geq 4}(A)$\textit{ has property }$(\mathrm{F}_{\mathcal{B}_{\mathrm{uc}}})$\textit{.}" The key fact is that the family of uniformly convex Banach spaces \textit{with uniform lower bounds for modulus of convexity} is  stable under ultraproducts \cite{AK}, and hence a similar argument to one in Remark~\ref{rem:exa} applies. 
\end{rem}

\section*{acknowledgments}
The author thanks his supervisor Narutaka Ozawa, and Nicolas Monod for comments and conversations. He also thanks Takeshi Katsura and Hiroki Sako, who first raised a question on $p$-Schatten classes to the author.  Finally, this work was carried out during a long stay (from February, 2010 to January, 2011) of the author at EPFL (\'{E}cole Polytechnique F\'{e}d\'{e}rale de Lausanne), on the Excellent Young Researcher Overseas Visiting Program by the Japan Society for the Promotion of Science. The author is grateful to Professor Nicolas Monod and his secretary Mrs.\ Marcia Gouffon for their warmhearted hospitality to his stay at EPFL.

\end{document}